\documentclass[12pt]{amsart}
\usepackage{graphicx}
\usepackage{amsmath,amssymb,verbatim,amscd,color}
\usepackage{subfig}

\usepackage{hyperref, todonotes}

\usepackage[alphabetic]{amsrefs}
\usepackage{xypic}
\usepackage{graphicx}
\usepackage{colordvi}
\usepackage{youngtab}

\numberwithin{equation}{section}

\newcommand\op{\operatorname}

\newcommand\mvv\Bbb

\newcommand\QH{\operatorname{QH}^*}

\newcommand\tensor{\otimes}

\newcommand\Gr{\operatorname{Gr}}

\newcommand{\ovop}[1]{\overline{\operatorname{#1}}}

\newtheorem{theorem}{Theorem}[section]
\newtheorem{remark}[theorem]{ Remark}

\newtheorem{corollary}[theorem]{Corollary}
\newtheorem*{corollary*}{Corollary}

\newtheorem{proposition}[theorem]{Proposition}
\newtheorem*{proposition*}{Proposition}

\newtheorem{lemma}[theorem]{Lemma}

\newtheorem{definition/lemma}[theorem]{Definition/Lemma}

\newcommand{\sL}{\mathfrak{sl}}

\newtheorem{example}[theorem]{\bf Example}

\begin{document}
\title{On $\operatorname{S}_n$-invariant conformal blocks vector bundles of rank one on $\operatorname{\overline M_{0,n}}$}
\author{Anna Kazanova}
\address{Department of Mathematics, University of Georgia, Athens, GA, 30602}
\email{kazanova@math.uga.edu}

\begin{abstract}For any simple Lie algebra,  a positive integer, and tuple of compatible weights, the conformal blocks bundle is a globally generated vector bundle on the moduli space of pointed rational curves. We classify all $S_n$-invariant vector bundles of conformal blocks for $\sL_n$ which have rank one.  We show that the cone generated by their base point free first Chern classes is polyhedral, generated by level one divisors.
\end{abstract}
\maketitle

\section{Introduction}
To any  simple Lie algebra $\mathfrak{g}$,  positive integer $\ell$, and  $n$-tuple $\vec{\lambda}$, of dominant weights for $\mathfrak{g}$ at level $\ell$, there is a globally generated vector bundle $\mathbb{V}({\mathfrak{g},\vec{\lambda},\ell})$ of conformal blocks on the moduli space $\ovop{M}_{0,n}$, of stable $n$-pointed rational curves \cites{TUY, Fakh}.  Their first Chern classes, the conformal blocks divisors $\mathbb{D}({\mathfrak{g},\vec{\lambda},\ell})$, are base point free, and therefore lie in the cone of nef divisors.  

Understanding the nef divisors on a variety is central to understanding its birational geometry. Vector bundles of conformal blocks and their Chern classes have been studied, primarily with standard intersection-theoretic methods, using Fakhruddin's formulas for the Chern classes and their intersections with F-curves~\cites{AGSS,Fed11,Swi11,BG12,Fakh, GG12,    Fed 13, Gia11,   GJMS12,AGS}.   Examples of conformal blocks divisors can be computed using Swinarski's implementation of these formulas into Macaulay2 software, \cite{ConfBlocks}.  While recursive,  and dependent on the computation of ranks of the bundles, computations are limited to  divisors of relatively low level on $\ovop{M}_{0,n}$ for low $n$.  Many open questions about the divisors persist. 
 
In this paper we study the subcone of the nef cone generated by an infinite set of divisors  $\mathcal{S}$,  consisting of the first Chern classes of conformal blocks vector bundles of rank 1 for $\sL_n$ with $S_n$-invariant  weights.   This is a generalization of \cite{AGSS}, where the authors studied a set of $\lfloor n/2\rfloor$  $S_n$-invariant $\sL_n$ divisors of level one, which are all
first Chern classes of rank one bundles.   In that paper, it was shown that each level one divisor spanned an extremal ray of the $S_n$-invariant nef cone of $\ovop{M}_{0,n}$.
While our family consists of infinitely many divisors, we prove that they all are contained in the cone generated by the original divisors studied in~\cite{AGSS}.

This work can be seen as an illustration for how to study nontrivial families of conformal block divisors on  $\operatorname{\overline M_{0,n}}$ using  Schubert calculus and tools from~\cite{bgm}. 

\medskip

 For the finite dimensional simple Lie algebra $\sL_{n}$, the dominant integral weights $\lambda$ are parameterized by Young diagrams $\lambda=(\lambda^{(1)}\geq \lambda^{(2)}\geq\dots\geq\lambda^{(n-1)}\geq \lambda^{(n)}\geq 0)$ with $\lambda^{(1)}\leq \ell$. We use the notation $|\lambda|=\sum_{i=1}^n \lambda^{(i)}$. As is standard, we denote the fundamental dominant weights of the form $\lambda =(1,\dots, 1, 0, \dots, 0)$  with $|\lambda| =i $  by $\omega_i$.

\medskip

Our main theorem provides a complete description of $S_n$-invariant rank 1 vector bundles for $\sL_n$ on $\ovop{M}_{0,n}$.

\begin{theorem} \label{Main}  Let $\Lambda=\{(\ell-m)\omega_i+m\omega_{i+1}:   0\leq m \leq \ell,\  0 \le i \le n-1\}$. Then we have
\begin{enumerate}

\item $\operatorname{rk}\mathbb{V}(\sL_n, \lambda^n,\ell)=1$ if  $\lambda\in \Lambda$.
\item  $\operatorname{rk}\mathbb{V}(\sL_n, \lambda^n,\ell)>1$ if   $\lambda\not\in \Lambda$.
\end{enumerate}
\end{theorem}
\noindent

\medskip

For any fixed $n>3$ we define  $\mathcal S$ to be the set of of all Chern classes of  rank 1 vector bundles, as classified by the Theorem~\ref{Main}:
$$\mathcal{S}=\{c_1\mathbb{V}({\sL_{n},((\ell-m)\omega_i+m\omega_{i+1})^{n},\ell}): \ell>0,  0 \le m \le \ell , 0 \le i \le n-1\}.$$

There are tools for studying rank one bundles (see Section \ref{RankOneTools}), which we use to give the following simple description as positive linear combinations of level one divisors.

\begin{proposition}\label{Decomposition} For any $\mathbb D \in \mathcal S$ we have the following decomposition:
 $$\mathbb{D}({\sL_{n},\!((\ell\!-\!m)\omega_i\!+\!m\omega_{i+1})^{n}\!,\ell})\!=\!(\ell-m)\mathbb{D}(\sL_{n},\!(\omega_i)^n,\!1)+ m \mathbb{D}(\sL_{n},(\omega_{i+1})^{n},1).$$
\end{proposition}

It therefore follows  that the cone generated by this infinite set of divisors $\mathcal{S}$ is in fact  polyhedral:

\begin{corollary} \label{LevelOne} The cone of divisors generated by  $S$ is the convex hull of the 
 $\lfloor \frac{n}{2}\rfloor-1$ extremal rays of the $\operatorname{S}_n$-invariant nef cone spanned by divisors $ \mathbb{D}({\sL_n, \omega_i^n, 1})$, where $2\leq i \leq \lfloor \frac{n}{2}\rfloor$.
 \end{corollary}

\subsection*{Acknowledgements} I am grateful to  Angela Gibney for many useful discussions, comments, and encouragement. I thank Prakash Belkale for pointing out the proof of the Lemma~\ref{bgmineq}, and for the comments on a draft of the paper. I also thank Dustin Cartwright and Linda Chen for helpful conversations.

 \section{Tools for computing ranks of conformal blocks bundles}\label{RankOneTools}

 We refer the reader to~\cite{bgm} for background information on conformal blocks vector bundles and divisors.
To compute the ranks of the vector bundles of conformal blocks, we use a special case of ``Witten's Dictionary", which is covered in Section~\ref{wd}.  The classical and quantum versions of the Pieri and Giambelli formulas are as often useful for applying Witten's dictionary, and we state those in Section~\ref{QGP}.

\subsection{Cohomological version of Witten's Dictionary}\label{wd} 
 Recall that the (small) quantum cohomology ring $\op{QH}^*(\op{Gr}(n,E);\mathbb{Z})$, for $E \cong \mathbb{C}^{n+\ell}$, is a $\mathbb{Z}[q]$-algebra isomorphic to 
$\op{H}^*(\op{Gr}(n,E);\mathbb{Z})\otimes_{\mathbb{Z}}\mathbb{Z}[q]$, as a module over $\mathbb{Z}[q]$, with elements  $\sigma_{\lambda}= \sigma_{\lambda}\tensor 1$,  where $\sigma_{\lambda} \in  \op{H}^*(\op{Gr}(n,E);\mathbb{Z})$ the cohomology class corresponding to the Schubert variety 
$\Omega_{\lambda}(F_{\bullet})$, where $F_{\bullet}$ is a full flag and $\lambda$ a partion (see for example \cite{Bertram} for the definitions).

 To compute the rank of our particular $S_n$-invariant conformal blocks bundles $\mathbb{V}({\sL_{n},\lambda^n,\ell})$, for $\sL_n$, one proceeds as follows.  For $|\lambda|=k$,
write $k=\ell+s$.  There are two cases:

  \begin{enumerate}
\item  If $s \le 0$,  then 
$\op{rk}\mathbb{V}({\sL_{n},\vec{\lambda},\ell}) $ is equal to the coefficient of the class of a point
 $\sigma_{k \omega_{n}}$  in the product  $\sigma_{\lambda_1} \cdots \sigma_{\lambda_n} \in \op{H}^*(\op{Gr}(n, \mathbb C^{n+k})).$

\item  If $s > 0$, then 
$\op{rk}\mathbb{V}({\sL_{n},\vec{\lambda},\ell})$ is  equal to the coefficient of the class of  $q^s[\mathrm{pt}]$ , where $([\mathrm{pt}]=\sigma_{\ell \omega_{n}}=\sigma_{(\ell,\ldots,\ell)})$ in the product  $\sigma_{\lambda_1} \star \cdots \star \sigma_{\lambda_n} \star \sigma_{\ell \omega_1}^s \in \op{QH}^*(\op{Gr}(n, E)).$

\end{enumerate}
See \cite[Theorem 3.6 and Remark 3.8]{b4} for the most general statement and  proof of Witten's Dictionary.

\subsection{Pieri and Giambelli formulas}\label{QGP}
 For convenience, we state  the classical and quantum versions of the  Pieri and Giambelli  formulas here.

\subsubsection*{Classical Pieri formula} If the Young diagram associated to $\lambda$ is contained in a $n\times \ell$ grid, and $i\leq n$, then the product of Schubert classes 
in the cohomology ring $\operatorname{H}^*\Gr(n,\ell+n)$ is given by
$$\sigma_{\lambda}\cdot \sigma_{p\omega_1} = \sum \sigma_{\pi},$$
where the sum is over all partitions $\pi$ obtained by adding $i$ boxes to $\lambda$, no two in the same column. 

\subsubsection*{Quantum Pieri formula}\cite{Bertram}  If the Young diagram associated to $\lambda$ is contained in a $n\times \ell$ grid, and $p\leq \ell$, then the product of Schubert classes 
in the quantum cohomology ring $\QH\Gr(n,\ell+n)$ is given by
$$\sigma_{\lambda}\star \sigma_{p\omega_1} = \sum \sigma_{\mu}+ q\sum \sigma_{\nu},$$
where the first sum is over all partitions $\mu$ obtained by adding $p$ boxes to $\lambda$, no two in the same column,  and the  second sum is over all partitions $\nu$ obtained by removing $n+\ell - p$ boxes from $\lambda$, at least one from each column.

\subsubsection*{Quantum Giambelli}\cite{Bertram} Set $\sigma_0=1$, $\sigma_i =0$ for $i< 0$ and for $i>\ell$.

If $\lambda$ is a partition contained in an $n \times \ell$ rectangle, then the Schubert class $\sigma_{\lambda} \in \QH \Gr(n, n+\ell)$ is given by 
$$\sigma_{\lambda} = \det(\sigma_{\lambda^{(i)}+j-i})_{1\leq i, j\leq n}.$$

\section{Rank one bundles}  
In this section we prove the first part of Theorem \ref{Main}, which states that the bundles $\mathbb{V}({\sL_{n}, (\lambda)^n,\ell}) $ have rank one if 
$\lambda\in \Lambda$. We will need the following Lemma.
 
\begin{lemma} 
\label{p:equalranks}\label{technical}
Let $\lambda = (\lambda^{(1)}, \dots, \lambda^{(n-1)}, 0)$ be such that $\lambda^{(1)}= \dots = \lambda^{(i)}=\ell$, $\lambda^{(i+1)}<\ell$. Denote by $\mu$ the partition $(\lambda^{(i+1)}, \dots, \lambda^{(n-1)}, 0, \dots, 0)$.
Then 
\begin{equation}
\op{rk}\mathbb{V}({\sL_{n}, \lambda^n,\ell})=\op{rk}\mathbb{V}({\sL_{n}, \mu^n,\ell}).
\end{equation}
\end{lemma}

\begin{proof} Using Giambelli formula, we can write $\lambda = \sigma_{\ell}^i \mu$ since

\begin{tiny}
$\lambda=  \left| \begin{matrix}
\sigma_{\ell}&0 &\dots& 0&0&\dots & 0&0\\
* & \sigma_{\ell} &\dots&0&0& \dots& 0&0 \\
 \dots& \dots& \dots&  \dots &  \dots & \dots &  \dots&  \dots\\
*& *& \dots &\sigma_{\ell}&  0 & \dots & 0&0\\

*& *& \dots & * & \sigma_{\lambda^{(i+1)}}&\dots & *&* \\
 \dots&  \dots& \dots &  \dots& \dots&\dots &  \dots& \dots \\
*& *& \dots & * & *&\dots & \sigma_{\lambda^{(n-1)}}& *\\
0 & 0 & \dots& 0& 0& \dots&0&  1

\end{matrix}\right| = \sigma_{\ell}^i
 \left| \begin{matrix}

 \sigma_{\lambda^{(i+1)}}&\dots & *&*&*& \dots & * \\
  \dots& \dots&\dots &  \dots& \dots &\dots &\dots   \\
 *&\dots & \sigma_{\lambda^{(n-1)}}& *& * & \dots & *\\
 0& \dots&0&  1& * & \dots & *\\
  0& \dots&0& 0 & 1  & \dots & *\\
 0& \dots&0&  0& 0  & \dots & 1\\
\end{matrix}\right| =  \sigma_{\ell}^i \mu.
$

\end{tiny}

We compute $\op{rk}\mathbb{V}({\sL_{n}, \lambda^n,\ell})$ using Witten's dictionary. 
Write $n |\lambda| = n \ell + n s$, so that $s = |\lambda|-\ell$,  and then $\op{rk}\mathbb{V}({\sL_{n}, \lambda^n,\ell})$ is equal to  a coefficient of $q^s [\mathrm{pt}] = q^s\sigma_{(\ell, \dots, \ell)}$ in the quantum product 
$\sigma_{\lambda}^{\star n}\star \sigma_{\ell}^s \in \op{QH} ^{*}(\Gr (n, n+\ell))).$
We have $$\sigma_{\lambda}^{\star n}\star \sigma_{\ell}^s= \sigma_{\ell}^{\star in}\star \sigma_{\mu}^{\star n}\star \sigma_{\ell}^s.$$
Write $|\lambda| = i\ell + |\mu|$, then $s = (i-1)\ell+ |\mu|$, so that $\sigma_{\lambda}^{\star n}\star \sigma_{\ell}^s = 
\sigma_{\ell}^{\star in}\star \sigma_{\mu}^{\star n}\star \sigma_{\ell}^{\star((i-1)\ell+ |\mu|)} = \sigma_{\ell}^{\star(in + (i-1)\ell + |\mu|)}\star \sigma_{\mu}^{\star n}.$
Note that $\sigma_{\ell}^{\star n} = \sigma_{(\ell, \dots, \ell)}$ and $\sigma_{\ell}^{\star (n+\ell)} = q^{\ell}\sigma_{(0, \dots, 0)}$ by quantum Pieri rule.  Thus, we have  $ \sigma_{\ell}^{\star(i-1)(\ell +n) } = q^{(i-1)\ell}\sigma_{(0, \dots, 0)}$.
Therefore 
$$\sigma_{\lambda}^{\star n}\star \sigma_{\ell}^s= q^{(i-1)\ell}\sigma_{\ell}^{\star(n + |\mu|)}\star \sigma_{\mu}^{\star n}.$$

If $|\mu|>\ell$, then $\sigma_{\lambda}^{\star n}\star \sigma_{\ell}^s = q^{i\ell}\sigma_{\ell}^{\star( |\mu|-\ell)}\star \sigma_{\mu}^{\star n}.$ In this case the coefficient of 
$q^s [\mathrm{pt}]$ in the quantum product 
$\sigma_{\lambda}^{\star n}\star \sigma_{\ell}^s$ is equal to the coefficient of $q^{|\mu|-\ell} [\mathrm{pt}]$ in the quantum product $\sigma_{\ell}^{\star( |\mu|-\ell)}\star \sigma_{\mu}^{\star n}   \in \op{QH} ^{*}(\Gr (n, n+\ell))).$ Note that the latter equals to the $\op{rk}\mathbb{V}({\sL_{n}, \mu^n,\ell})$ by Witten's dictionary.

If $|\mu|\leq \ell$, then $\sigma_{\lambda}^{\star n}\star \sigma_{\ell}^s = q^{i(\ell-1)}\sigma_{\ell}^{\star( n+ |\mu|)}\star \sigma_{\mu}^{\star n}.$ Since by the Pieri rule, $\sigma_{\ell}^{\star( n+ |\mu|)} = q^{|\mu|}\sigma_{(\ell-|\mu|, \dots, \ell-|\mu| )}$, we conclude that $\sigma_{\lambda}^{\star n}\star \sigma_{\ell}^s = q^{i(\ell-1)+ |\mu|}\sigma_{(\ell-|\mu|, \dots, \ell-|\mu| )} \star \sigma_{\mu}^{\star n} = q^s\sigma_{(\ell-|\mu|, \dots, \ell-|\mu| )} \star \sigma_{\mu}^{\star n}$. In this case the coefficient of 
$q^s [\mathrm{pt}]$ in the quantum product 
$\sigma_{\lambda}^{\star n}\star \sigma_{\ell}^s$ is equal to the coefficient of $\sigma_{(\ell, \dots, \ell)}$ in the classical product $\sigma_{(\ell-|\mu|, \dots, \ell-|\mu| )} \cdot \sigma_{\mu}^{ n} \in \op{H}^* \Gr(n, n+\ell)$. Note that this coefficient is  equal to the coefficient of $\sigma_{(|\mu|, \dots, |\mu|)}$ in the classical product $\sigma_{\mu}^{ n} \in \op{H}^* \Gr(n, n+|\mu|)$, which is equal to  $\op{rk}\mathbb{V}({\sL_{n}, \mu^n,\ell})$ by Witten's dictionary. 
\end{proof}

\begin{proposition}
\label{allrank1}
$\mathbb{V}({\sL_{n}, ((\ell-m)\omega_i+m\omega_{i+1})^n,\ell})$ are rank one bundles. 
\end{proposition}
\begin{proof} By Proposition~\ref{p:equalranks}, we have $ \op{rk}\mathbb{V}({\sL_{n},((\ell-m)\omega_i+m\omega_{i+1})^n,\ell} )= \op{rk}\mathbb{V}({\sL_{n},(m \omega_1)^n,\ell} )$.

Since $n\cdot |m\omega_1| = nm\leq n \ell$, by Witten's dictionary, $\op{rk}\mathbb{V}({\sL_{n},(m \omega_1)^n,\ell} )$ is equal to the multiplicity of class of the point $\sigma_{(m, \dots, m)}$ in the product $\sigma_{m}^n\in H^* \Gr(n, n+m)$. Since $\sigma_{m}^n = \sigma_{(m, \dots, m)}$ by Pieri rule, we have $\op{rk}\mathbb{V}({\sL_{n},(m \omega_1)^n,\ell} )=1$, and we conclude that the rank  $ \op{rk}\mathbb{V}({\sL_{n},((\ell-m)\omega_i+m\omega_{i+1})^n,\ell} )$ is equal to 1.
\end{proof}

\section{Higher rank bundles}  

Recall  from the Introduction that $\Lambda$ is the set of weights $\{(\ell-m)\omega_i+m\omega_{i+1}:   0\leq m \leq \ell,\  0 \le i \le n-1\}$.  In Proposition~\ref{not1}, we will show that $\operatorname{rk}\mathbb{V}({\sL_n,\mu^n,\ell})>1$, for $\mu\not\in \Lambda$. Combined with Proposition~\ref{allrank1}, this  completes the proof of Theorem~\ref{Main}. We begin with some special cases, and the main proof will reduce to the special cases.

\begin{lemma} We have  \label{tinyrank}$\op{rk}\mathbb{V}({\sL_{n},(\omega_i )^n, 2})>1$ for $1< i <  n-1$ and $n\geq 4$.
\end{lemma}
\begin{proof} Since $\op{rk}\mathbb{V}({\sL_{n},(\omega_i )^n, 2} )= \op{rk}\mathbb{V}({\sL_{n},(\omega_{n-i} )^n, 2})$, without loss of generality we may assume that $i\leq \lfloor n/2\rfloor$.

Suppose that $i=2$. By Witten's dictionary, we have to find the coefficient of $\sigma_{(2,\dots,2)}$ in the classical product $\sigma_{\omega_2}^n$.  Note that $\sigma_{\omega_2}^2 = \sigma_{\omega_4}+ \sigma_{\omega_3+\omega_1}+ \sigma_{2\omega_2}$. For each of these terms we have $\sigma_{\omega_2}^{n-2} \cdot \sigma_{\omega_4} = a_1\sigma_{(2,\dots,2)} $,  $\sigma_{\omega_2}^{n-2} \cdot\sigma_{\omega_3+\omega_1} = a_2\sigma_{(2,\dots,2)} $, and  $\sigma_{\omega_2}^{n-2} \cdot \sigma_{2\omega_2} = a_3\sigma_{(2,\dots,2)} $, where each of the constants $a_1$, $a_2$, and $a_3$ are at least 1. Thus $\op{rk}\mathbb{V}({\sL_{n},(\omega_2 )^n, 2})\geq 3$.

Now assume that $i>2$. Since $ni = 2n + (i-2)n$,  by Witten's dictionary, we need to compute the coefficient of $q \sigma_{(2,\dots,2)}$ in the quantum  product $\sigma_{\omega_i}^{\star n}\star \sigma_{2\omega_1}^{i-2}$.

Write $n = iq+r$, where $0\leq r<i$, let $\alpha = q+1$. Then $\sigma_{\omega_i}^{\star n}\star \sigma_{2\omega_1}^{i-2} = 
\sigma_{\omega_i}^{n-\alpha(i-2)} \star \sigma_{\omega_i}^{\star \alpha(i-2)}\star \sigma_{2\omega_1}^{i-2}$, and
 $ \sigma_{\omega_i}^{\star \alpha}\star \sigma_{2\omega_1}  = c \sigma_{\omega_n + \omega_{i-r}}\star \sigma_{2\omega_1}+\mbox{other terms} = c q \sigma_{\omega_{r-i}} + \mbox{other  terms}.$
Note that since $\alpha\geq 3$, we always have $c\geq 2$, and all the other terms have nonnegative coefficients, since a product of effective cycles is effective.
Thus $\sigma_{\omega_i}^{\star n}\star \sigma_{2\omega_1}^{i-2} = c^{i-2}q^{i-2}\sigma_{\omega_i}^{n-\alpha(i-2)} \star \sigma_{\omega_{r-i}}^{i-2}+ \mbox{other  terms}.$ 

Since $\sigma_{\omega_i}^{n-\alpha(i-2)} \star \sigma_{\omega_{r-i}}^{i-2} = \sigma_{(2, \dots, 2)}+\mbox{ other terms} $, we conclude that $\op{rk}\mathbb{V}({\sL_{n},(\omega_i )^n, 2})\geq c^{i-2}$, and in particular,  we have $\op{rk}\mathbb{V}({\sL_{n},(\omega_i )^n, 2})>1$.
\end{proof}

\begin{lemma}\label{bgmineq}
We have  $\op{rk}\mathbb{V}({\sL_{n}, ((a \omega_i + b \omega_{i+1} )^n, a+b+1})>1$ for all $a\geq 0$, $b>0$, $1<i<n-1$.
\end{lemma}
\begin{proof} By Proposition~\ref{allrank1}, we have $\op{rk}\mathbb{V}({\sL_{n}, (a \omega_i + (b-1) \omega_{i+1} )^n, a+b-1})=1$. Therefore  by~\cite[Prop. 17.1]{bgm}, the map
\begin{equation*}
\mathbb{V}({\sL_{n}, (a \omega_i + b \omega_{i+1} )^n, a+b+1})\to \mathbb{V}({\sL_{n}, (a \omega_i + (b-1) \omega_{i+1} )^n, a+b-1})\otimes \mathbb{V}({\sL_{n}, ( \omega_{i+1}  )^n, 2})
\end{equation*}
is a surjection, and the rank $\op{rk} \mathbb{V}({\sL_{n}, (a \omega_i + b \omega_{i+1} )^n, a+b+1})$ is greater than or equal to the rank $\op{rk}  \mathbb{V}({\sL_{n}, ( \omega_{i+1}  )^n, 2})$. By Lemma~\ref{tinyrank}, we know that $\op{rk}\mathbb{V}({\sL_{n},(\omega_{i+1} )^n, 2})>1$.
\end{proof}

\begin{proposition}\label{not1} 
We have 
$\operatorname{rk}\mathbb{V}({\sL_n,\mu^n,\ell})>1$,
 if $\mu \not\in \Lambda$.
\end{proposition}
\begin{proof} Throughout the proof we will use the following classical formula (see for example~\cite[Lemma 1.8]{bgm}). For all  $c>0$, we have
\begin{equation}\label{ranklemma}
\op{rk}\mathbb{V}({\sL_{n}, \mu^n,\ell})\leq  \op{rk}\mathbb{V}({\sL_{n}, \mu^n,\ell+c}).
\end{equation}

\medskip

\begin{enumerate}

\item
First, suppose that $\mu = b\omega_i$. Then since $\mu\not\in \Lambda$, we have  $1<i<n-1$, and $b<\ell$. In this case,  
$\op{rk}\mathbb{V}({\sL_{n}, \mu^n,\ell})\geq  \op{rk}\mathbb{V}({\sL_{n}, (b\omega_i)^n, b+1})$ by~(\ref{ranklemma}). We have  $\op{rk}\mathbb{V}({\sL_{n}, (b\omega_i)^n, b+1})>1$ by Lemma~\ref{bgmineq}.

\item
Second, suppose that $\mu = a\omega_i + b\omega_j$, where $i<j$,  $a, b\neq 0$ and $ \ell \geq (a+b)$. 

\begin{enumerate}
\item
If $\ell = a+b$, then since $\mu \not\in \Lambda$, we have $j\geq  i+2$. By Theorem~\ref{technical}, $\op{rk}\mathbb{V}({\sL_{n}, \mu^n,\ell})=  \op{rk}\mathbb{V}({\sL_{n}, (b\omega_j)^n, a+b})$. Now we reduced the problem to the case (1).

\item
If $\ell > a+b$, and $j > i+1$, we can apply~(\ref{ranklemma}) to conclude that 
$\op{rk}\mathbb{V}({\sL_{n}, \mu^n,\ell})\geq  \op{rk}\mathbb{V}({\sL_{n}, (a\omega_i+ b\omega_j)^n, a+b})$, so we reduced to the case (2i).

\item If $\ell > a+b$, and $j=i+1$, we can apply~(\ref{ranklemma}) to conclude that 
$\op{rk}\mathbb{V}({\sL_{n}, \mu^n,\ell})\geq  \op{rk}\mathbb{V}({\sL_{n}, (a\omega_i+ b\omega_{i+1})^n, a+b+1})$. We have $\op{rk}\mathbb{V}({\sL_{n}, (a\omega_i+ b\omega_{i+1})^n,\! a\!+\!b\!+\!1})>1$ by Lemma~\ref{bgmineq}.

\end{enumerate}

\item Finally,  let $i_1<\dots< i_k$ be an ordered subset of  $\{1, \dots, n-1\}$, and let $\mu = \sum_{j=1}^k c_{i_j} \omega_{i_j}$ with all $c_{i_j}>0$, and $\ell(\mu)=\sum_{j=1}^k c_{i_j} \leq  \ell$, and $k\geq 3$.   By~(\ref{ranklemma})  we have
\begin{equation*}\label{lowerlevel}
\op{rk}\mathbb{V}({\sL_{n}, \mu^n,\ell})\geq  \op{rk}\mathbb{V}({\sL_{n}, \mu^n,\ell(\mu)}).
\end{equation*} 
Using Proposition~\ref{p:equalranks}, we obtain that
\begin{equation*}
\op{rk}\mathbb{V}({\sL_{n},\mu^n,\ell(\mu)} )= \op{rk}\mathbb{V}({\sL_{n}, (\sum_{j=2}^{k}c_{i_j} \omega_{i_j})^n,\ell(\mu)}), 
\end{equation*}
so that
\begin{equation*}
\op{rk}\mathbb{V}({\sL_{n},\mu^n,\ell} )\geq  \op{rk}\mathbb{V}({\sL_{n}, (\sum_{j=2}^{k}c_{i_j} \omega_{i_j})^n,\ell(\mu)}). 
\end{equation*}
We can repeat this process $k-2$ times to conclude that
\begin{equation*}
\op{rk}\mathbb{V}({\sL_{n}, \mu^n,\ell}) \geq  \op{rk}\mathbb{V}({\sL_{n}, (c_{i_{k-1}} \omega_{i_{k-1}} + c_{i_k} \omega_{i_k} )^n,\ c_{i_{k-2}}+c_{i_{k-1}}+ c_{i_{k}}}).
\end{equation*}
Thus we reduced to the case (2ii) if $i_k>i_{k-1}+1$, and the case (2iii) if $i_k = i_{k-1}+1$.

\end{enumerate}

\end{proof}

\begin{example}In this example we elaborate on the computation in the proof of Lemma~\ref{tinyrank} to show  that $\op{rk}\mathbb{V}({\sL_7, \omega_3^7,2})\geq 2.$
Using Witten's dictionary, we need to compute the coefficient of $q \sigma_{(2,2,2,2,2,2,2)}$ in the quantum  product $\sigma_{\omega_3}^{\star 7}\star \sigma_{2\omega_1}$.

We see that $\alpha = 3$, and 

{\Yvcentermath1
$$\sigma_{\omega_3}^{\star 3} = 3\ \ {\tiny \yng(2,2,1,1,1,1,1)} + \mbox{ other effective terms}.$$}

Then by Pieri rule, $\sigma_{\omega_3}^{\star 3}\star \sigma_{2\omega_1} = 3q \ \ {\tiny \yng(1,1)}$ + $\mbox{ other effective  terms}.$
{\Yvcentermath1
 $\mbox {Finally, \ \ }\sigma_{\omega_3}^{\star 7}\star \sigma_{2\omega_1} = 3q \ \  {\tiny \yng(2,2,2,2,2,2,2)}+ \mbox{ other effective  terms}.$}
In particular, $\op{rk}\mathbb{V}({\sL_7, \omega_3^7,2})\geq 3.$

\end{example}

\section{Decomposition of $\mathbb D({\sL_{n},((\ell-m)\omega_i+m\omega_{i+1})^n,\ell})$}

In this section we identify all $\mathbb D \in \mathcal S$ as effective sums of level one divisors, and from this conclude that the cone $\mathcal S$ is finitely generated.

\begin{proof}[Proof of Proposition~\ref{Decomposition}]
By Proposition~\ref{allrank1}, we have $\op{rk}\mathbb{V}({\sL_n, (\ell-m)\omega_i+m\omega_{i+1},\ell})=1$. 
Moreover, since we know that the level one bundles 
$\mathbb{V}({\sL_n, \omega_j^n,1})$ have rank one, and so by Belkale's quantum generalization of Fulton's conjecture~\cite{Bel07}, $\op{rk}\mathbb{V}({\sL_n, (N\omega_j)^n,N})=1$ for all integers $N$.
We have that  $\op{rk}\mathbb{V}({\sL_n,( (\ell-m)\omega_i)^n,(\ell-m)})=1$ and $\op{rk}\mathbb{V}({\sL_n, (m\omega_{i+1})^n,m})=1$.  
So by applying \cite[Prop. 17.1]{bgm} and \cite[Cor. 17.3]{bgm}, we conclude that 
\begin{equation*}
\begin{aligned}
\mathbb{D}({\sL_{n},((\ell-m)\omega_i+m\omega_{i+1})^n,\ell})&=\mathbb{D}({\sL_{n},(\ell-m)\omega_i^n,\ell-m})+  \mathbb{D}({\sL_{n},m\omega_{i+1}^n,m})\\
&=(\ell-m)\mathbb{D}({\sL_{n},\omega_i^n,1})+ m \mathbb{D}({\sL_{n},\omega_{i+1}^n,1}).
\end{aligned}
\end{equation*}
\end{proof}

\begin{remark}\label{rkdim}
 Without loss of generality assume that $m>0$.
Using Macaulay2, we checked that up to $n=2000$, each family $$\mathcal{F}^n_{\ell,m}=\{\mathbb{D}(\sL_{n},((\ell-m)\omega_i+m\omega_{i+1})^n,\ell): 1 \le i \le \lfloor n/2\rfloor -1\}$$ gives a basis of $\operatorname{Pic}(\ovop{M}_{0,n})^{\operatorname{S}_n}$ by intersecting divisors from the family $\mathcal{F}^n_{\ell,m}$ with the independent set  $\{F_{1,1,i, n-i-2}\}_{i=1}^{\lfloor n/2\rfloor -1}$ of F--curves on $(\ovop{M}_{0,n})^{\operatorname{S}_n}$ using~\cite[Proposition 5.2]{Fakh}.
\end{remark}

We checked the full dimensionality of $\mathcal S$ for all $n\leq 2000$, and sporadically for some larger $n$, and we  believe that the statement of Remark~\ref{rkdim} is true for all $n$.

So at least up to $n = 2000$, the cone generated by $\mathcal{S}$ is full dimensional as it contains all the full dimensional cones generated by the  $\mathcal{F}^n_{\ell,m}$.   By Proposition \ref{Decomposition}, each divisor $\mathbb D\in \mathcal S$ is a linear combination of $\mathbb{D}({\sL_{n},\omega_{i}^n,1})$, which for $2 \le i \le \lfloor n/2 \rfloor$, by \cite{AGSS} define extremal rays of the $\operatorname{S}_n$-invariant nef cone $\operatorname{Nef}(\ovop{M}_{0,n})^{\operatorname{S}_n}$.  So the cone generated by $\mathcal{S}$ is equal to the cone spanned by these rays.


\begin{bibdiv}
\begin{biblist}


\bib{AGS}{article}{
  author={Alexeev, Valery},
  author={Gibney, Angela},
  author={Swinarski, David},
  title={Higher level conformal blocks on $\overline {\operatorname {M}}_{0,n}$ from $\sL _2$},
  journal={Proc. Edinb. Math. Soc.,},
   volume={57},
  year={2014},
  pages={7-30},
 
}

\bib{AGSS}{article}{
  author={Arap, Maxim},
  author={Gibney, Angela},
  author={Stankewicz, Jim},
  author={Swinarski, David},
  title={$\sL _n$ level $1$ Conformal blocks divisors on $\overline {\operatorname {M}}_{0,n}$},
  journal={International Math Research Notices},
  date={2011},
}


\bib{Bel07}{article}
{
    AUTHOR = {Belkale, Prakash},
     TITLE = {Geometric proof of a conjecture of {F}ulton},
   JOURNAL = {Adv. Math.},
  FJOURNAL = {Advances in Mathematics},
    VOLUME = {216},
      YEAR = {2007},
    NUMBER = {1},
     PAGES = {346--357},
    }

\bib{b4}{article}{
    AUTHOR = {Belkale, Prakash},
     TITLE = {Quantum generalization of the {H}orn conjecture},
   JOURNAL = {J. Amer. Math. Soc.},
  FJOURNAL = {Journal of the American Mathematical Society},
    VOLUME = {21},
      YEAR = {2008},
    NUMBER = {2},
     PAGES = {365--408},
}	


\bib{bgm}{article}{
   author={Belkale, Prakash},
   author={Gibney, Angela},
   author={Mukhopadhyay, Swarnava},
   title={Quantum cohomology and conformal blocks on $\overline M_{0,n}$},
   date={2013},
    note={arXiv:1308.4906v3 [math.AG]}
}	




\bib{Bertram}{article}{
author = {Bertram, Aaron},
    title = {Quantum Schubert Calculus},
    journal = {Adv. Math},
    year = {1997},
    volume = {128},
    pages = {289--305}
    }
    
  

\bib{BG12}{article}{
  author={Bolognesi, Michelle},
 author={Giansiracusa, Noah},
  title={Factorization of point configurations, cyclic covers, and conformal blocks},
 journal={J. Eur. Math. Soc, to appear,},
  date={2012},
  note={arXiv:1208.4019 [math.AG]},
}




\bib{Fakh}{article}{
   author={Fakhruddin, Najmuddin},
   title={Chern classes of conformal blocks},
   conference={
      title={Compact moduli spaces and vector bundles},
   },
   book={
      series={Contemp. Math.},
      volume={564},
      publisher={Amer. Math. Soc.},
      place={Providence, RI},
   },
   date={2012},
   pages={145--176},
}

\bib{Fed11}{article}{
  author={Fedorchuk, Maksym},
  title={Cyclic Covering Morphisms on $\bar {M}_{0,n}$},
  date={2011},
  note={arXiv:1105.0655 [math.AG]},
}

\bib{Fed13}{article}{
  author={Fedorchuk, Maksym},
  title={New nef divisors on $\overline{M}_{0,n}$},
  date={2013},
  note={arXiv:1308.5993 [math.AG]},
}




\bib{Gia11}{article}{
  author={Giansiracusa, Noah},
  title={Conformal blocks and rational normal curves},
  journal={Journal of Algebraic Geometry,},
  year={2013},
  volume={22},
  pages={773--793},
}


\bib{GG12}{article}{
  author={Giansiracusa, Noah},
  author={Gibney, Angela},
  title={The cone of type A, level 1 conformal block divisors},
  journal={Adv. Math.},
  volume={231},
  pages={798--814},
  year={2012},
}


\bib{GJMS12}{article}{
  author={Gibney, Angela},
  author={Jensen, David},
  author={Moon, Han-Bom},
  author={Swinarski, David},
  title={Veronese quotient models of $\overline{\operatorname{M}}_{0,n}$ and conformal blocks},
  journal={Michigan Math Journal},
  volume={62},
  pages={721--751},
  year={2013},
}


\bib{ConfBlocks}{article}{
		author={Swinarski, David},
		title={\texttt{\upshape ConformalBlocks}: a Macaulay2 package for computing conformal block divisors},
		date={2010},
		note={Version 1.1, {http://www.math.uiuc.edu/Macaulay2/}},
}

\bib{Swi11}{article}{
		author={Swinarski, David},
		title={$sl_2$ conformal block divisors and the nef cone of $\bar{M}_{0,n}$},
		date={2011},
		journal={Exp. Math, to appear,},
		note={arXiv:1107.5331 [math.AG]},
}
\bib{TUY}{article} {
    AUTHOR = {Tsuchiya, Akihiro},
    AUTHOR = {Ueno, Kenji},
    AUTHOR = {Yamada, Yasuhiko},
     TITLE = {Conformal field theory on universal family of stable curves
              with gauge symmetries},
 BOOKTITLE = {Integrable systems in quantum field theory and statistical
              mechanics},
    SERIES = {Adv. Stud. Pure Math.},
    VOLUME = {19},
     PAGES = {459--566},
 PUBLISHER = {Academic Press},
   ADDRESS = {Boston, MA},
      YEAR = {1989},  
}
\end{biblist}
\end{bibdiv}
\end{document}